\documentclass[12pt,reqno]{amsart}

\usepackage{amsmath,amsthm,amssymb,comment,fullpage}
\usepackage{braket}
\usepackage{mathtools}
\usepackage{flexisym}
\usepackage{breqn}
\usepackage{amssymb}
\usepackage{mathrsfs}
\DeclareMathAlphabet{\mathzc}{OT1}{pzc}{m}{it}
\usepackage{latexsym}
\usepackage{amssymb}

\usepackage{caption}
\usepackage{times}
\usepackage[T1]{fontenc}
\usepackage{mathrsfs}
\usepackage{svg}
\usepackage{latexsym}
\usepackage{epsfig}
\usepackage{amsmath,amsfonts,amsthm,amssymb,amscd}
\input amssym.def
\input amssym.tex
\usepackage{color}
\usepackage{hyperref}
\usepackage{url}
\newcommand{\bburl}[1]{\textcolor{blue}{\url{#1}}}

\usepackage{tikz}
\usepackage{tkz-tab}
\usetikzlibrary{shapes.geometric,positioning}

\usepackage{scalerel}

\def\vecsign#1{\rule[1.388\LMex]{\dimexpr#1-2.5pt}{.36\LMpt}%
  \kern-6.0\LMpt\mathchar"017E}

\usepackage{stackengine,amsmath}
\stackMath
\usepackage{graphicx}

\newcommand{\burl}[1]{\textcolor{blue}{\url{#1}}}

\newcommand{\mattwo}[4]
{\left(\begin{array}{cc}
                        #1  & #2   \\
                        #3 &  #4
 \end{array}\right) }

\numberwithin{equation}{section}

\newtheorem{thm}{Theorem}[section]

\theoremstyle{plain}

\newtheorem{definition}[thm]{Definition}

\newtheorem{lemma}[thm]{Lemma}

\newtheorem{theorem}[thm]{Theorem}

\theoremstyle{definition}

\newcommand\be{\begin{equation}}
\newcommand\ee{\end{equation}}
\newcommand\bee{\begin{equation*}}
\newcommand\eee{\end{equation*}}
\newcommand\bea{\begin{eqnarray}}
\newcommand\eea{\end{eqnarray}}
\newcommand\beae{\begin{eqnarray*}}
\newcommand\eeae{\end{eqnarray*}}
\newcommand\bi{\begin{itemize}}
\newcommand\ei{\end{itemize}}
\newcommand\ben{\begin{enumerate}}
\newcommand\een{\end{enumerate}}
\newcommand\bc{\begin{center}}
\newcommand\ec{\end{center}}
\newcommand\ba{\begin{array}}
\newcommand\ea{\end{array}}

\newcommand{\Z}{\ensuremath{\mathbb{Z}}}
\newcommand{\Q}{\mathbb{Q}}
\newcommand{\N}{\mathbb{N}}

\newcommand\frakfamily{\usefont{U}{yfrak}{m}{n}}
\DeclareTextFontCommand{\textfrak}{\frakfamily}

\newcommand\leg[2]{{#1\overwithdelims () #2}}

\newtheorem{rek}[thm]{Remark}

\newcommand{\hr}[1]{\href{#1}{\url{#1}}}

\newsavebox\EllipticCurve
\savebox{\EllipticCurve}(200,180)[cc]{%
\setlength{\unitlength}{.6pt}
\begin{picture}(200,180)
\thicklines
\put(0,90){\begin{picture}(200,90)
   \qbezier(18,0)(18,22.5)(54,22.5)
   \qbezier(54,22.5)(69.75,20.25)(85.5,18)
   \qbezier(85.5,18)(101.25,13.5)(117,13.5)
   \qbezier(117,13.5)(153,13.5) (189,85.5)
\end{picture}
} \put(0,90){\begin{picture}(200,90)
   \qbezier(18,0)(18,-22.5)(54,-22.5)
   \qbezier(54,-22.5)(69.75,-20.25)(85.5,-18)
   \qbezier(85.5,-18)(101.25,-13.5)(117,-13.5)
   \qbezier(117,-13.5)(153,-13.5) (189,-85.5)
\end{picture}
}
\end{picture}
}

\title{Connections of Class Numbers to the Group Structure of Generalized Pythagorean Triples}
\author{Thomas Jaklitsch}
\email{\textcolor{blue}{\href{mailto:thomasjaklitsch@college.harvard.edu}{thomasjaklitsch@college.harvard.edu}}}
\address{Department of Mathematics, Harvard University, Cambridge, MA 02138}

\author{Thomas C. Martinez}
\email{\textcolor{blue}{\href{mailto:tmartinez@hmc.edu}{tmartinez@hmc.edu}}}
\address{Department of Mathematics, Harvey Mudd College, Claremont, CA 91711}

\author{Steven J. Miller}
\email{\textcolor{blue}{\href{mailto:sjm1@williams.edu}{sjm1@williams.edu}},  \textcolor{blue}{\href{Steven.Miller.MC.96@aya.yale.edu}{Steven.Miller.MC.96@aya.yale.edu}}}
\address{Department of Mathematics and Statistics, Williams College, Williamstown, MA 01267}

\author{Sagnik Mukherjee}
\email{\textcolor{blue}{\href{mailto:smukherjee@cmi.ac.in}{smukherjee@cmi.ac.in}}}
\address{Department of Mathematics, Chennai Mathematical Institute, Siruseri, Chennai, Tamilnadu-603103}

\thanks{We thank Amnon Yekutieli for introducing us to the problem and sharing his work on the subject, and our colleagues from the 2021 Polymath REU, Leart Ajvazaj, Manyi Guo, Dylan Jamner, Yuan Lu, Jonathan Marvel-Zuccola, Sydney Morgan, and Bangqi (Blair) Yuan, for numerous helpful conversations and comments, especially Sydney and Blair, who worked with the third named author on a preliminary research project which was the springboard for this work. Those lectures, notes and questions were a valuable starting point for the present paper.}



\date{\today}

\begin{document}

\maketitle

\begin{abstract} Two well-studied Diophantine equations are those of Pythagorean triples and elliptic curves; for the first we have a
parametrization through rational points on the unit circle, and for the second we have a structure theorem for the group of rational solutions.
Recently Yekutieli discussed a connection between these two problems, and described the group structure of Pythagorean triples and the
number of triples for a given hypotenuse. We generalize these methods and results to Pell's equation. We find a similar group structure and count
on the number of solutions for a given $z$ to $x^2 + Dy^2 = z^2$ when $D$ is 1 or 2 modulo 4 and the class group of $\mathbb{Q}[\sqrt{-D}]$ is a free
$\Z_2$ module, which always happens if the class number is at most 2. We give examples of when the results hold for a
class number greater than 2, as well as an example with different behavior when the class group does not have this structure.
\end{abstract}


\tableofcontents

\section{Introduction}

The study of the number and structure of rational solutions to Diophantine equations (polynomials of finite degree with integer coefficients) is related to numerous important problems in mathematics, from Pythagorean triples to elliptic curves. Much is known for these two problems, where we can parametrize the solutions, which form commutative groups; see for example \cite{Kn, Maz1, Maz2, MT-B, ST}. We generalize these results to Pell's equation $x^2 + Dy^2 = z^2$, and show that for certain $D$, leading to class groups where every element has order at most 2, we have similar group structures

Our motivation is a recent paper by Yekutieli \cite{Ye}. The Pythagorean triples are integer solutions of the equation $x^2 + y^2 = z^2$, and correspond to rational points on the unit circle; thus to a triple $(a,b,c)$ we associate the complex number \be \zeta_{a,b,c} \ =\ x + iy\ =\ \frac{a}{c} + \frac{b}{c}i.\ee These can be parametrized by looking at lines with rational slope emanating from a fixed rational point, often taken to be $(-1, 0)$. There are four solutions where either $a$ or $b$ is zero: $1, i, -1, -i$. These are the units of $\Z[i] = \{a + ib: a, b \in \Z\}$, and correspond to trivial Pythagorean triples. We now consider $\zeta$ where both $a$ and $b$ are non-zero. We cannot have $a=b$, as that would lead to $\sqrt{2}$ being rational. A straightforward calculation shows that given such a solution $\zeta$ there are seven other distinct conjugate solutions; we can multiply $\zeta$ by $i, i^2$ and $i^3$ (the units of $\Z[i]$ other than 1) and then we can take the complex conjugates of these four solutions. We illustrate this in Figure \ref{fig:complexpythagoras}; note without loss of generality given any Pythagorean triple not associated to a unit of $\Z[i]$ we may always adjust it, through multiplication by a unit and complex conjugation if needed, so that it lies in the shaded region (i.e., the second octant, or the part of the first quadrant where the imaginary part exceeds the real part).

\begin{figure}[h]
\begin{center}
\scalebox{1}{\includegraphics{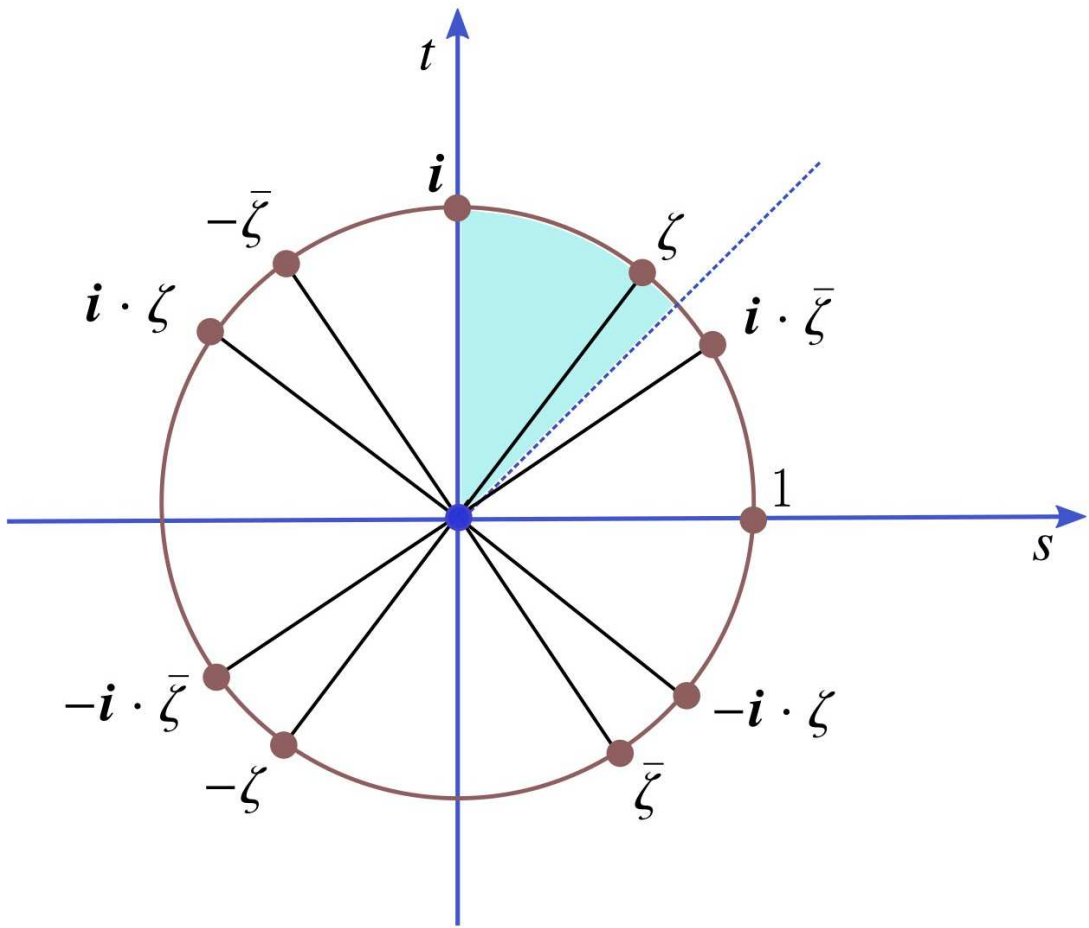}}
\caption{\label{fig:complexpythagoras} The four trivial solutions ($1, i, -1, -i$) and the seven conjugates to a non-trivial solution $\zeta$, which can be taken to lie in the second octant. Image from \cite{Ye}.}
\end{center}\end{figure}

Identifying Pythagorean triples with complex numbers yields a commutative group through the multiplicativity of the norm. While rescaling a Pythagorean triple by $k$ does not change the complex number associated to it, multiplying associated complex numbers (or raising one to a power) generates new solutions. For example, the triple $(3,4,5)$ yields $\zeta_{3,4,5} = 3/5 + i 4/5$, and \be \zeta_{3,4,5}^2 \ = \ \left(\frac34 + \frac45 i\right) \left(\frac34 + \frac45 i\right) \ = \ -\frac{7}{25} + \frac{24}{25}i,\ee  which corresponds to the triple $(7, 24, 25)$, while \be  \zeta_{3,4,5} \ \ \zeta_{5,12,13} \ = \ \left(\frac35 + \frac45 i\right) \left(\frac{5}{13} + \frac{12}{13} i\right) \ = \ -\frac{33}{65} + \frac{56}{65} i, \ee  which corresponds to the triple $(33, 56, 65)$.


Yekutieli \cite{Ye} proved several results about the structure of the group of rational solutions to the unit circle version of the Pythagorean equation. Specifically, denote these solutions by \begin{equation} G(\Q) \ := \ \{x + iy: x, y \in \Q \ {\rm and}\ x^2 + y^2 = 1\}. \end{equation} This is a group under complex multiplication, and decomposes as \begin{equation} G(\Q) \ = \ U \times F, \end{equation} where $U = \{1, i, -1, -i\}$ is the units in $\Z[i]$ and $F$ is a free abelian group with basis given by the collection $\{\zeta_p\}_{p \in P_1}$, where the primes $P$ decompose as \begin{equation} P \ = \ P_1 \sqcup P_2 \sqcup P_3, \ \ \ {\rm with} \ \ \ P_{\ell} \ := \ \{p \in P: p \equiv \ell \bmod 4\}.\end{equation} He then proves results on which $c$ yield Pythagorean triples, and how many there are.

Our goal is to generalize these results and describe the structure of solutions to $x^2 + D y^2 = z^2$ for square-free $D$ (there is no loss in generality in having a positive sign, as $x^2 + D y^2 = z^2$ is the same as $z^2 - D y^2 = x^2$). In particular, we are interested in seeing how the structure of $\Z[\sqrt{-D}]$ influences the solutions; one way to measure this structure is through its class number. The proofs in \cite{Ye} crucially use that $\Z[i] = \Z[\sqrt{-1}]$ has class number 1. There are 8 other square-free $D$ such that $\Z[\sqrt{-D}]$ has class number 1; the complete set (see \cite{Wa}) is \be -D \ = \ \{-1,\ \ -2,\ \ -3,\ \ -7,\ \ -11,\ \ -19,\ \ -43,\ \ -67,\ \ -163\}.\ee 

For suitably restricted $D$, we can generalize the method in \cite{Ye}. First we define \textit{normalized solutions} to an arbitrary Pell equation as follows.

\begin{definition}\label{def:normsoln}
A solution $(a,b,c)$ with $a,b,c\in \N$ to 
\begin{equation}
    x^2+Dy^2\ = \ z^2 \label{eq:pelleq}
\end{equation} is defined to be a \textbf{normalized solution} if $\gcd(a,b,c)=1$.
\end{definition}

We also define
\be G_D(\Q) \ :=\ \{a + b\sqrt{-D} \in \Q[\sqrt{-D}] : a^2 + Db^2 = 1\}, \ee 
and, with some restrictions on $D$, prove that $G_D(\Q) = U\times F$ where $U :=\{1,-1\}$ and $F$ is a free abelian group, which allows us to determine the number of normalized solutions of the form $(a,b,c)$ to the equation \eqref{eq:pelleq} for any given $c\in \N$.

We do this by deriving three theorems which describe the factorization of elements in $G_D(\mathbb{Q})$ and how it relates to the number of normalized solutions of the equation $x^2 + Dy^2 = c^2$. Our generalization depends on properties of the class group, which leads to restrictions on what $D$ we can analyze. 

Generalization from $D=1$ to an arbitrary $D>0$ is difficult as the ring $\Z[\sqrt{-D}]$ is not necessarily a unique factorization domain, or even a Dedekind domain, and hence the factorization of the elements of $\Z[\sqrt{-D}]$ into primes or irreducibles can be complicated (and sometimes not possible). Thus unlike the case of $D=1$, the factorization of the elements of $G_D(\Q)$ are no longer automatically inherited from the factorization of the elements of $\Z[\sqrt{-D}]$.


We recall some definitions and results on class groups; see Chapters 2 and 3 of \cite{Cox} for details. Given a $K<0$, the class group $C(K)$ is the set of equivalence classes $P/\sim$ together with the operation of Dirichlet composition,  where
\be P \ := \ \{\text{primitive, positive-definite forms with discriminant K}\}.\ee  For two binary quadratic forms $f=[a,b,c]$ and $g=[a',b',c']$, $f \sim g$ if and only if there exists a matrix $A =\mattwo{p}{q}{r}{s} \in \text{SL}_2(\Z)$ such that $g(x,y) = f^A(x,y) := f(px+qy,rx+sy)$. The identity element of this class is
\begin{equation}
\text{Identity of C(K)} \ = \
    \begin{cases}
        \text{the class containing }[1,0,\frac{-K}{4}], & \text{if } K\equiv 0\ (\text{mod 4}),\\
        \text{the class containing }[1,1,\frac{1-K}{4}], & \text{if } K\equiv 1\ (\text{mod 4}).
    \end{cases}
\end{equation}
The inverse of the class containing $[a,b,c]$ is the class containing $[a,-b,c]$.\\ \

The following is our main result.

\begin{thm}\label{thm:mainresult}  Assume $-D \equiv 2$ or $3$ $(\text{mod } 4)$ and $D > 1$. Suppose the class group of $\mathbb{Q}[\sqrt{-D}]$ is a free $\Z_2$-module.
Then $G_D(\Q)=U\times F$, where $U=\{\pm 1\}$ and $F$ is a free abelian group. If $c = p_1^{n_1}\cdots p_k^{n_k}$ such that $\leg{-D}{p_i}  = 1$ for all $1\le i \le k$, then the number of normalized solutions of the form $(a,b,c)$ is $2^{k-1}$. Otherwise, there are no normalized solutions of the form $(a, b, c)$. \end{thm}

\begin{rek} As the case for $D = 1$ is known (see \cite{Ye}), we only consider $D > 1$. It is worth noting that the argument for $D = 1$ is slightly different, because the group of units for $\Z[i]$ is $\{\pm 1, \pm i\}$ while the group of units for $\Z[\sqrt{-D}]$ when $D > 1$ is just $\{\pm 1\}$. Also, the definition of normalized solutions must be altered to account for the fact that if $(a,b,c)$ is a solution then so is $(b,a,c)$, which is not true for $D > 1$. See Remark $\ref{rek: D=1}$ for greater detail on the case of $D = 1$. 
\end{rek}

After recalling needed facts, we show that if we assume the hypotheses of Theorem $\ref{thm:mainresult}$, then $c$ is a normalized solution to the equation $x^2 + Dy^2$. The fact that each element of the class group has order at most 2 is crucial here for the following reason. The integer $c$ is properly represented by some binary quadratic form. Therefore, if every element of $C(-4D)$ has order at most 2, then $c^2$ is properly represented by the identity element of $C(-4D)$, which is the form $x^2 + Dy^2$. So, there exists a normalized solution to the equation $(a,b,c)$ for some $a,b \in \Z$. If it is the case that there exists $f \in C(-4D)$ such that $|f| > 2$, then there may exist an integer $c$ such that $c$ is represented by $f$ but $c^2$ is not represented by the identity element. In this case $c$ might not be a solution to the equation even though it satisfies $\left(\frac{-D}{p}\right) = 1$ for all prime factors $p$ of $c$. We refer to the remarks at Section 5 for a concrete example.

Next, we prove that $G_D(\mathbb{Q})$ factors into the direct product of the group of units of $\Z[\sqrt{-D}]$ and a free abelian group. Finally, this factorization allows us to determine the number of solutions to $x^2 + Dy^2 = z^2$ for a fixed integer $z$. We conclude with examples of these theorems, as well as cases where the theorem does not hold (e.g., when $C(-4D)$ is not a free $\Z_2$-module).

\section{Units and Complex Multiplication}

We start by determining the group operations for \be G_D(\Q)\ := \ \{z=a+b\sqrt{-D}\in \Q[\sqrt{-D}]: a^2+Db^2=1\}.\ee
The norm on $\Q[\sqrt{-D}]$ is \be N(a+b\sqrt{-D})\ :=\ a^2+Db^2\ee and is multiplicative: if $a_1 + b_1\sqrt{-D}, a_2 + b_2\sqrt{-D} \in G_D(\mathbb{Q})$, then  \be N\left((a_1 + b_1\sqrt{-D})\cdot (a_2 + b_2\sqrt{-D})\right)\ =\ N\left(a_1 + b_1\sqrt{-D}\right) \cdot N\left(a_2 + b_2\sqrt{-D}\right)\ =\ 1.\ee  Thus the product is an element of $G_D(\Q)$. The inverse of any $a + b\sqrt{-D} \in G_D(\Q)$ is given by $\frac{a - b\sqrt{-D}}{a^2+Db^2}$. The group of units for the ring $\Z[\sqrt{-D}]$ are those elements $x + y\sqrt{-D}$ where $x,y \in \Z$ such that $N(x + y\sqrt{-D}) = 1$. This corresponds to those elements such that $x^2 + Dy^2 = 1$. As $D > 1$ the only integer solutions to this equation are $\{\pm 1\}$, so the group of units is $U = \{\pm 1\}$.

The group $G_D(\Q)$ can be geometrically viewed as the rational points on the ellipse with rational co-ordinate $x^2+Dy^2=1$. Given two points $(x_1,y_1)$ and $(x_2,y_2)$ on this ellipse, we can multiply them as follows:
\begin{equation}
    (x_1,y_1)*(x_2,y_2)\ := \ (x_1x_2-Dy_1y_2,\ x_1y_2+x_2y_1), \label{Elliptic Multiplication}
\end{equation} which yields another rational point on this ellipse. Note that each such rational point on this ellipse corresponds to a unique normalized solution to the equation $x^2+Dy^2=z^2$ up to the sign, since by definition normalized solutions are positive. Our aim is to find \textit{elementary normalized solutions} so that we can generate any normalized solution by multiplying (as above) these elementary normalized solutions, similar to building composite numbers by multiplying prime numbers.

As multiplying normalized solutions using the above rule does not always yield a normalized solution, we work with the \textit{elementary solutions}. We see this in detail while studying Lemma \ref{lem:multsols} and Theorem \ref{thm:retrievesols}.

\section{Group Structure on the Rational Solutions on Ellipse}

In this section we give a \textit{structure theorem} for the set \be G_D(\Q)\ := \ \{z=a+b\sqrt{-D}\in \Q[\sqrt{-D}]: a^2+Db^2=1\}.\ee 
We prove that $G_D(\Q)$ is of the form $U\times F$ where $U:=\{1,-1\}$ and $F$ is a free abelian group provided that $-D\equiv 2,3\ (\text{mod }4)$ and $C(-4D)$ is a free $\Z_2$-module.

We begin with the following results necessary for factoring $G_D(\mathbb{Q})$.

\subsection{Conditions for Existence of Solutions}

\begin{lemma}
\label{lem:oddsols}
Given a $D>0$ such that $-D\equiv 2,3\ (\text{mod }4)$, if $(a,b,c)$ is a normalized solution to
\begin{equation}
    x^2+Dy^2\ = \ z^2, \label{eq:pelleqrepeated}
\end{equation}
then $c$ must be an odd natural number.\par\vspace{0.15cm}
\end{lemma}

\begin{proof} We have $a^2+Db^2=c^2$. Assume $2\mid c$. Then $4\mid a^2+Db^2$. Note that $a^2\equiv 0,1\ (\text{mod }4)$ and $b^2 \equiv 0,1\ (\text{mod }4)$. Therefore, since $4\mid a^2+Db^2$, we have $b^2 \equiv a^2 \equiv 0\ (\text{mod }4)$. This implies $2\mid a, b, c$, which contradicts the fact $(a,b,c)$ is a normalized solution. 
\end{proof}

Since we are focusing our attention on the case when $-D\equiv 2,3\ (\text{mod }4)$, by Lemma $\ref{lem:oddsols}$ we are only concerned with normalized solutions $(a,b,c)$ when $c$ is odd.

The next result determines when $x^2 + Dy^2 = z^2$ has a normalized solution for fixed $z$. In all arguments below $\leg{a}{p}$ represents the Legendre symbol; it is 1 if $a$ is a non-zero square modulo $p$, 0 if $a$ is congruent to zero modulo $p$, and -1 otherwise.
    
The following lemma is important for determining when a normalized solution to (\ref{eq:pelleqrepeated}) exists. We use Hensel's Lemma to prove it, and for completeness, we state it below.

\begin{lemma}[Hensel's Lemma]\label{lem:hensel}
    Let $f(x)$ be a polynomial with integer coefficients. Let $k$ be a positive integer, and $r$ an integer such that $f(r) \equiv 0$ (mod $p^k$). Suppose $m\le k$ is a positive integer. Then, if $f'(r) \ne 0$ (mod $p$), there is an integer $s$ such that $f(s) \equiv 0$ (mod $p^{k+m})$ and $s \equiv r$ (mod $p^k$). 
\end{lemma}

\begin{lemma}\label{lem:solexistence}
Suppose $C(-4D) \cong (\Z_2)^n$ for some $n\geq 0$ such that $-D = 2,3\ (\text{mod }4)$ . Let $c = p_1^{n_1}\cdots p_k^{n_k}$ be an odd positive integer. There exists a normalized solution $(a,b,c)$ to $x^2 + Dy^2 = z^2$ if and only if $\leg{-D}{p_i}=1$ for $1 \le i \le k$.
\end{lemma}

\begin{proof}  First suppose $\leg{-D}{p_i} =1$ for all $1 \le i\le k$. We need to show that $-D$ is a quadratic residue modulo $c$. Fix some prime $p_i$ and consider the polynomial $f(x) = x^2 + D$. Since $-D$ is a non-zero quadratic residue modulo $p_i$, we know that there exists an $r \in \Z$ such that 
$r\ne 0$ and $f(r) = r^2 + D \equiv 0$ (mod $p_i$). Also, $f'(x) = 2x$, so since $p_i$ is an odd prime, if $f'(r) = 2r \equiv 0$ (mod $p_i$), then $r \equiv 0$ (mod $p_i$). This implies that $D \equiv 0$ (mod $p_i$), which is a contradiction. Therefore, we can apply Hensel's Lemma to obtain some $s \in \Z$ such that $f(s) = s^2 + D \equiv 0$ (mod $p_i^2$). Now suppose that for some $k < n_i$, and some $s_k \in \Z$ such that $s_k \equiv r$ (mod $p_i$) we have that $f(s_k) = s_k^2 + D \equiv 0$ (mod $p_i^k$). Then we have that $f'(s_k) = 2 s_k \ne 0$ (mod $p_i^k$), since $s_k \equiv r$ (mod $p_i$) and $r \ne 0$ (mod $p_i$). If $s_k \equiv 0$ (mod $p_i^k)$, then $s_k \equiv 0$ (mod $p_i$), which is a contradiction. Therefore, by applying Hensel's Lemma again, we obtain $s_{k+1} \in \Z$ such that $f(s_{k+1}) = (s_{k+1})^2 + D \equiv 0$ (mod $p^{k+1})$, and $s_{k+1} \equiv s_k$ (mod $p_i^k$), so $s_{k+1} \equiv r$ (mod $p_i$). By this inductive process, we get a solution to $f(x) = x^2 + D\equiv 0$ (mod $p_i^{n_i}$). This implies that $-D$ is a quadratic residue modulo $p_i^{n_i}$.

For each $i \ne j$, $p_i^{n_i}$ and $p_j^{n_j}$ are coprime. Also, $-D$ is a quadratic residue modulo $p_i^{n_i}$ and $-D$ is a quadratic residue modulo $p_j^{n_j}$ by the above. Therefore, $-D$ is a quadratic residue modulo $p_i^{n_i}p_j^{n_j}$. Thus, $-4D$ is a quadratic residue modulo $c$.  Then by Lemma 2.5 of \cite{Cox}, $c$ is properly represented by a primitive form $f$ of discriminant $-4D$. Now if we apply Lemma 2.3 of \cite{Cox}, we get $f \sim [c, \alpha, \beta]$ for some integers $\alpha, \beta$.

Note that $\alpha^2 - 4c\beta = -4D$. Suppose for contradiction that $(\alpha, c) > 1$. Then this implies that $(\alpha, c) | c$ and $(\alpha, c) | D$. However, since $\leg{-D}{p_i} =1$ for all $i\le k$, $c$ and $D$ are coprime. This is a contradiction, so $(\alpha, c) = 1$. Therefore, $(c, c, \frac{2\alpha}{2}) = 1$, and hence the Dirichlet composition is well defined in this case, and we we have that $f^2 \sim [c^2, \alpha', \beta']$ for $\alpha',\beta'\in \Z$. This implies $f^2$ properly represents $c^2$, since $C(-4D) \cong (\Z_2)^n$ for some $n$, so each element of $C(-4D)$ has order at most $2$. Therefore, the class $[f]$ in $C(-4D)$ has order at most $2$, so $[f]^2$ is the identity and hence $f^2\sim [1,0,D]$. Thus, by the above, $[1,0,D] \sim [c^2, \alpha', \beta']$. From this, using Lemma 2.3 of \cite{Cox} once again, we infer that $x^2 + Dy^2$ properly represents $c^2$. Therefore, there exists a normalized solution $(a,b,c)$ to \eqref{eq:pelleqrepeated} for some $a,b \in \Z$. This completes the first implication. 

Conversely, let $(a,b,c)$ be a normalized solution to \eqref{eq:pelleqrepeated}. First, suppose for contradiction that $(c,D) = h_0 > 1$. Then we have $h_0 | a^2$. Suppose $h_0 = q_1^{m_1} \cdots q_r^{m_r}$ for primes $q_i$. Define $h_1 := q_1\cdots q_r$. This gives us $h_1 > 1$ and $h_1 \mid c$, $h_1 \mid h_0$, and $h_1 \mid a$. Therefore, $h_1^2 | a^2$ and $h_1^2 | c^2$. This implies that $h_1^2 \mid Db^2$ but $h_1^2$ does not divide $D$, because $D$ is square free. Thus, $h_1$ divides $b$, and therefore $(a, b, c) > 1$, which is a contradiction. Hence $(c,D) = 1$.

Let $p$ be a prime factor of $c$. We claim that $(b,p)=1$. Suppose not. Then since $p\mid b$, and  $p\mid c$ we get $p \mid a$. This
contradicts the assumption that $(a,b,c) = 1$, and thus $(b, p) = 1$ as claimed. Thus $a^2+Db^2=c^2 $ implies $a^2+Db^2\equiv 0\ (\text{mod }p)$ which
gives us $-D\equiv(ab^{-1})^2\ (\text{mod }p)$ where $b^{-1}$ is the inverse of $b$ modulo $p$ which exists as $(b,p)=1$.
So, $\left(\frac{-D}{p}\right)=1$ for any prime factor $p$ of $c$.
\end{proof}

\subsection{Factorization of $G_D(\mathbb{Q})$:} 

In order to state and prove the theorems which provide a factorization of $G_D(\Q)$, we first need the following two lemmas.

\begin{lemma}\label{lem:division}
For some $p$ such that $\leg{-D}{p_i}  = 1$, let $x_0^2 + Dy_0^2 = p^{2\alpha}$ where $(x_0, y_0) = 1$. Also assume $p^{2\alpha} | c^2 + Dd^2$ for some $c,d$ such that $(c,d) = 1$. Then \textit{exactly one} of the following is true:
\be x_0+y_0\sqrt{-D}\mid c+d\sqrt{-D}\ee  or \be x_0-y_0\sqrt{-D}\mid c+d\sqrt{-D}\ee  in $\Z[\sqrt{-D}]$.
\end{lemma}

\begin{proof} First, we have $p^{2\alpha} \mid x_0^2 + Dy_0^2$, $p^{2\alpha} \mid c^2 + Dd^2$.  Therefore, $p^{2\alpha} \mid x_0^2 d^2 - c^2 y_0^2$ which implies $ p^{2\alpha} \mid (dx_0 + cy_0)(dx_0 - cy_0)$.

We claim that $p^{2\alpha}$ divides \textit{exactly one} of $dx_0+cy_0$ and $dx_0-cy_0$. If not, and $p$ divides both $(dx_0 + cy_0)$ and $(dx_0 - cy_0)$, then $p$ divides $2dx_0$. Since $p$ is odd, $p \mid dx_0$. If $p \mid d$ then $p \mid c$, which is a contradiction, while  if $p \mid x_0$, then $p \mid Dy_0^2$. However, if $p \mid y_0$, then $(x_0, y_0) > 1$. Therefore, $p \mid D$, contradicting Lemma \ref{lem:solexistence}.

Therefore, $p^{2\alpha}$ divides \textit{exactly one} of $(dx_0 + cy_0)$ and $(dx_0 - cy_0)$. Now let us look at these two cases separately.

\vspace{0.15cm}

\textbf{Case 1: $p^{2\alpha} \mid (dx_0 - cy_0)$.} In this case we have
\be p^{4\alpha}\mid (x_0^2+Dy_0^2)(c^2+Dd^2) \ =\  (cx_0+Ddy_0)^2+D(dx_0-cy_0)^2,\ee 
and since $p^{4\alpha} \mid (dx_0 - cy_0)^2$ by assumption, we get 
$ p^{4\alpha}\mid (cx_0+Ddy_0)^2$. Therefore, $p^{2\alpha}\mid (cx_0+Ddy_0)$. 

Consider the number $\alpha+\beta\sqrt{-D}$ where $\alpha = \frac{cx_0+Ddy_0}{p^{2\alpha}}$ and $\beta = \frac{dx_0-cy_0}{p^{2\alpha}}$. Note that $$(\alpha+\beta\sqrt{-D})(x_0+y_0\sqrt{-D}) \ = \ (c+d\sqrt{-D}).$$ Therefore, $(x_0+y_0\sqrt{-D})\mid  (c+d\sqrt{-D})$ in $\Z[\sqrt{-D}]$. 

\vspace{0.2cm}

\textbf{Case 2: $p^{2\alpha} | (dx_0 + cy_0)$.} We proceed in a similar fashion and show that
\be (\gamma+\delta\sqrt{-D})(x_0-y_0\sqrt{-D})\ = \ (c+d\sqrt{-D}),\ee 
where $\gamma = \frac{cx_0-Ddy_0}{p^{2\alpha}}$ and $\delta = \frac{dx_0+cy_0}{p^{2\alpha}}$.
\end{proof}

\vspace{0.5cm}

\begin{lemma}\label{lem:definezeta}
    Let $p$ be an odd prime. Then \ref{eq:pelleq} has unique normalized solution of the form $(a,b,p)$.
\end{lemma}
\begin{proof}
    Let $a,b,c,d$ be integers. We first want to show, if $a + b\sqrt{-D} \mid c + d\sqrt{-D}$ and $c + d\sqrt{-D} \mid a + b\sqrt{-D}$ in 
    $\Z[\sqrt{-D}]$, then $c = \pm a$ and $b = \pm d$. 
    
    We can take $(\frac{a + b\sqrt{-D}}{c + d\sqrt{-D}}) = x + y \sqrt{-D}$ and $(\frac{c + d\sqrt{-D}}{a + b\sqrt{-D}}) = w + z\sqrt{-D}$. then we get $(x + y\sqrt{-D})(w + z\sqrt{-D}) = 1$. This implies $(x + y\sqrt{-D}), (w + z\sqrt{-D})$ are units, but the only units in $\mathbb{Z}[\sqrt{-D}]$ are $\pm 1$. Therefore, we get $z = y = 0$ and $x = \pm 1$, $w = \pm 1$. Therefore, $a = \pm b$ and $c = \pm d$.
    
    Now suppose that $p^2 = x_0^2 + Dy_0^2 = x_1^2 + Dy_1^2$ such that $(x_0, y_0) = (x_1, y_1) = 1$. This implies, by Lemma $\ref{lem:solexistence}$ that $\leg{-D}{p} = 1$. Then by Lemma $\ref{lem:division}$ we have one of the following four cases:
    \begin{enumerate}
        \item $x_0 + y_0 \sqrt{-D} \mid x_1 + y_1 \sqrt{-D}$ and $x_1 + y_1 \sqrt{-D} \mid x_0 + y_0 \sqrt{-D}$,
        \item $x_0 + y_0 \sqrt{-D} \mid x_1 + y_1 \sqrt{-D}$ and $x_1 - y_1 \sqrt{-D} \mid x_0 + y_0 \sqrt{-D}$, 
        \item $x_0 - y_0 \sqrt{-D} \mid x_1 + y_1 \sqrt{-D}$ and $x_1 + y_1 \sqrt{-D} \mid x_0 + y_0 \sqrt{-D}$, or
        \item $x_0 - y_0 \sqrt{-D} \mid x_1 + y_1 \sqrt{-D}$ and $x_1 - y_1 \sqrt{-D} \mid x_0 + y_0 \sqrt{-D}$.
    \end{enumerate}
    
    The proof of the lemma for Case (1) follows directly from the above argument. For Case (2), we have $x_0 - y_0 \sqrt{-D} \mid x_1 - y_1 \sqrt{-D} \mid x_0 + y_0 \sqrt{-D}$. Therefore, $x_0 + y_0 \sqrt{-D}| x_0 - y_0 \sqrt{-D}$ and $x_0 - y_0 \sqrt{-D}| x_0 + y_0 \sqrt{-D}$, so we get that $x_0 = -x_0$ or $y_0 = - y_0$, which implies $x_0 = \pm p$, since $x_0$ cannot be $0$. Case (3) is the same argument as Case (2). For Case (4) we have $x_0 - y_0 \sqrt{-D} \mid x_1 + y_1 \sqrt{-D}$ and $x_1 + y_1 \sqrt{-D} \mid x_0 - y_0 \sqrt{-D}$, so this case is equivalent to Case (1). 
\end{proof} 

Now, let us define the set
\[
    S\ := \ \left\{\text{odd prime } q : \leg{-D}{q} = 1\right\}.
\]
For each $q\in S$, we also define
\be \zeta_q \ := \  \frac{x_0+y_0\sqrt{-D}}{q},\ee 
where $q^2=x_0^2+Dy_0^2$ and $x_0, y_0 > 0$.

Note that $x_0, y_0$ exist due to Lemma $\ref{lem:solexistence}$ by the definition of the set $S$ and they are unique by Lemma $\ref{lem:definezeta}$. These $\zeta_q$'s are the so-called \textit{elementary solutions}. Our objective is to determine a one to one correspondence between the products of powers of $\zeta_q$'s and the set of normalized solutions of the form $(a,b,c)$ for a given $c$.


Now we are ready to state and prove the first theorem.

\begin{theorem}
\label{thm:factorgdq}
Let $z = \frac{a + b\sqrt{-D}}{c} \in G_D(\mathbb{Q})$ where $(a,b) = 1$, $c > 1$. Let $c = p_1^{\alpha_1} \cdots p_k^{\alpha_k}$. Then \be z \ = \  \pm \zeta_{p_1}^{\pm \alpha_1}\cdots \zeta_{p_k}^{\pm \alpha_k}.\ee 
\end{theorem}

\begin{proof}
Note that given such a $z$, we have  $a^2+Db^2=c^2$ and thus $(a,b,c)$ is a normalized solution to $x^2 + Dy^2 = z^2$. Thus according to the Lemma \ref{lem:solexistence}, for all the prime factors $q$ of $c$, $\leg{-D}{q} = 1$. Thus  $\zeta_{p_i}$ is well defined.

Consider $z$ as in the statement. Then
\begin{gather*}
    a^2+Db^2\ = \ c^2\ = \ p_1^{2\alpha_1}p_2^{2\alpha_2}\cdots p_k^{2\alpha_k}.
\end{gather*}
Note that each $p_i^2 = \gamma_i^2+D\beta_i^2$ for some $\gamma_i,\beta_i\in \Z$ such that $(\gamma_i, \beta_i) = 1$.
Therefore,
\be p_i^{2\alpha_i}\ = \ (\gamma_i^2+D\beta_i^2)^{\alpha_i} \ = \ \left(\gamma_i+\beta_i\sqrt{-D}\right)^{\alpha_i}\left(\gamma_i-\beta_i\sqrt{-D}\right)^{\alpha_i} \ := \  (x_i+y_i\sqrt{-D})(x_i-y_i\sqrt{-D})\ee 
for each $i$. Here we used the fact that product of numbers of the form $x^2+Dy^2$ is again of the form $x^2+Dy^2$
Thus for each $i$, \be p_i^{2\alpha_i}\ = \ x_i^2+Dy_i^2\ee  and $p_i^{2\alpha_i}\mid a^2+Db^2$. Hence by Lemma \ref{lem:division}, we have exactly one of the following:
\be  x_i+y_i\sqrt{-D}\mid a+b\sqrt{-D} \ \Leftrightarrow\ x_i-y_i\sqrt{-D}\mid a-b\sqrt{-D}\ee 
or
\be  x_i-y_i\sqrt{-D}\mid a+b\sqrt{-D} \ \Leftrightarrow\ x_i+y_i\sqrt{-D}\mid a-b\sqrt{-D}.\ee 

As $$a^2+Db^2 \ = \  c^2\ = \ p_1^{2\alpha_1}p_2^{2\alpha_2}\cdots p_k^{2\alpha_k},$$ we have
\begin{dmath}
    (a+b\sqrt{-D})(a-b\sqrt{-D}) \ = \  (x_1+y_1\sqrt{-D})(x_2+y_2\sqrt{-D})\cdots (x_k+y_k\sqrt{-D})\cdot  (x_1-y_1\sqrt{-D})(x_2-y_2\sqrt{-D})\cdots (x_k-y_k\sqrt{-D}),
\end{dmath}
which implies that
\begin{dmath}
    (a_1+b_1\sqrt{-D})(a_1-b_1\sqrt{-D}) \ = \ (x_2+y_2\sqrt{-D})\cdots (x_k+y_k\sqrt{-D})\cdot (x_2-y_2\sqrt{-D})\cdots (x_k-y_k\sqrt{-D})
\end{dmath}
where
\begin{equation}
a_1+b_1\sqrt{-D} \ := \
\begin{cases}
\frac{a+b\sqrt{-D}}{x_1+y_1\sqrt{-D}}, & \text{if } x_1+y_1\sqrt{-D}\mid a+b\sqrt{-D} \\ \\
\frac{a+b\sqrt{-D}}{x_1-y_1\sqrt{-D}}, & \text{if } x_1-y_1\sqrt{-D}\mid a+b\sqrt{-D}. \\
\end{cases}
\end{equation}

Also note that $a_1^2+Db_1^2 = p_2^{2\alpha_2}\cdots p_k^{2\alpha_k}$. If we continue this process inductively and keep defining subsequent terms $a_i+b_i\sqrt{-D}$ as above, at the final step we get the following \be (a_k+b_k\sqrt{-D})(a_k-b_k\sqrt{-D})\ = \ 1\ee  where $a_k,b_k\in \Z$. But that would mean $a_k=\pm 1, b_k=0$. Thus,
\be (a+b\sqrt{-D})\ = \  \pm 1\cdot (x_1\pm y_1\sqrt{-D})(x_2\pm y_2\sqrt{-D})\cdots (x_k\pm y_k\sqrt{-D}),\ee  and dividing by $c$, we obtain
\be \frac{a+b\sqrt{-D}}{c}\ = \ \pm 1\cdot
\frac{x_1\pm y_1\sqrt{-D}}{p_1^{\alpha_1}}\cdot\frac{x_2\pm y_2\sqrt{-D}}{p_2^{\alpha_2}} \cdots  \frac{x_k\pm y_k\sqrt{-D}}{p_k^{\alpha_k}}.\ee 
Now note that
\be \zeta_{p_i}^{\pm \alpha_i}\ = \ \left(\frac{\gamma_i+ \beta_i\sqrt{-D}}{p_i}\right)^{\pm\alpha_i}\ = \ \left(\frac{x_i\pm y_i\sqrt{-D}}{p_i^{\alpha_i}}\right).\ee 
\end{proof}


\subsection{Obtaining Solutions from Factorization}

We explore consequences of being able to factor every element of $G_D(\Q)$. Given a factorization of some $z\in G_D(\Q)$, we determine $c$ where $(a,b,c)$ is the normalized solution to \eqref{eq:pelleqrepeated} corresponding to $z$.

We first prove a needed result.

\begin{lemma}
\label{lem:multsols}
Let $z_1 = \frac{a_1 + b_1 \sqrt{-D}}{c_1} \in G_D(\Q)$ and $z_2 = \frac{a_2 + b_2\sqrt{-D}}{c_2} \in G_D(\Q)$ such that $(a_1, b_1) = (a_2, b_2) = (c_1, c_2) = 1$. Then \be (a_1a_2-Db_1b_2,a_1b_2+a_2b_1, c_1c_2)\ee  is a normalized solution, and $a_1a_2-Db_1b_2$ and $a_1b_2+a_2b_1$ are co-prime.
\end{lemma}

\begin{proof} It suffices to show that there are no common prime divisor of $a_1a_2-Db_1b_2$ and $a_1b_2+a_2b_1$. Assume for contradiction that there exists some prime $q$ that divides them. We then have $q\mid a_1a_2-Db_1b_2$ and $q\mid a_1b_2+a_2b_1$, and  \begin{center}
$q\mid (-b_2)(a_1a_2-Db_1b_2) + (a_2)(a_1b_2+a_2b_1)$ \end{center} implies $ q\mid b_1(a_2^2+Db_2^2)$. However if $q\mid b_1$ then $q\mid a_1a_2$ and $ q\mid a_1b_2$. Since $(a_1, b_1) = 1$ and $q \mid b_1$, we have $q \nmid a_1$. Therefore, $q\mid a_2 \text{ and } q\mid b_2$. This is a contradiction, and thus $q\mid (a_2^2+Db_2^2)$.

Also note that $q | (a_1)(a_1a_2-Db_1b_2) + (Db_1)(a_1b_2+a_2b_1)= a_2(a_1^2+Db_1^2)$ implies $ q\mid a_2, \text{ or } q\mid a_1^2+Db_1^2$. Now assume that $q\mid a_2$. Then $q\mid Db_1b_2$ and $q\mid a_1b_2$ which implies $ q\mid a_1$, and we have $q\nmid b_2$. Suppose for contradiction that $q\mid D$. Then $q\mid a_2^2+Db_2^2=c_2^2$ and since $q$ is a prime, $q^2\mid c_2^2$ implies $ q^2\mid Db_2^2$ which yields $ q^2\mid D$ as $q\nmid b_2$. This contradicts that $D$ is square-free.

Thus $q\mid a_2$ implies $ q\mid b_1 \text{ and } q\mid a_1$, which contradicts that $(a_2,b_2)=1$. Therefore $q\nmid a_2$ implies $ q\mid a_1^2+Db_2^2$, and $q$ is a common divisor of $c_1$ and $c_2$, which is a contradiction.

Hence there cannot be a common prime factor of $a_1a_2-Db_1b_2\text{ and }a_1b_2+a_2b_1$. \end{proof}

\begin{rek} Note that this proof does not depend on the fact that the class group is a free $\Z_2$-module. So, given two normalized solutions $(a,b,c)$ and $(a',b',c')$ such that $(c,c') = 1$, one can multiply them together to get a new normalized solution of the form $(x,y,cc')$. \end{rek}

Now, we state the theorem that retrieves solutions from the factorization of an element in $G_D(\Q)$.

\begin{theorem}\label{thm:retrievesols}
For some $z \in G_D(\mathbb{Q})$ let $z = \pm\zeta_{p_1}^{\pm \alpha_1} \cdots \zeta_{p_k}^{\pm \alpha_k}$ where $\leg{-D}{p_i} = 1$ for all $i$. Then if $z$ corresponds to the normalized triple $(a, b, c)$, we must have $c = p_1^{\alpha_1} \cdots p_k^{\alpha_k}$.
\end{theorem}

\begin{proof} By definition of $G_D(\mathbb{Q})$, $z$ corresponds to the normalized triple $(a,b,c)$ if and only if $z=\frac{a+b\sqrt{-D}}{c}$ where $(a,b)=1$, so we write $z=\frac{a+b\sqrt{-D}}{c}$ where $(a,b)=1$.
Due to Lemma \ref{lem:solexistence}, for each $p_i$, we have $\zeta_{p_i}^{\pm \alpha_i} = \frac{a_i + b_i \sqrt{-D}}{p_i^{\alpha_i}}$ where $(a_i, b_i) = 1$, since we can write $p_i^{2\alpha_i} = a_i^2 + Db_i^2$ where $(a_i, b_i) = 1$. By Lemma \ref{lem:multsols},  $\pm \zeta_{p_1}^{\pm\alpha_1}\cdot\zeta_{p_2}^{\pm\alpha_2}\cdots \zeta_{p_k}^{\pm\alpha_k}$ corresponds to a normalized solution of the form $(\alpha, \beta, p_1^{\alpha_1}p_2^{\alpha_2}\cdots p_k^{\alpha_k})$, $z$ already corresponds to the normalized solution $(a,b,c)$ and we know that an element of $G_D(\mathbb{Q})$ can correspond to just one normalized solution. So we must have $c = p_1^{\alpha_1}\cdots p_k^{\alpha_k}$.\end{proof}

We have now proved two of our main results. First, Theorem \ref{thm:factorgdq}, which factorizes elements of $G_D(\Q)$, and second, Theorem \ref{thm:retrievesols}, which retrieves normalized solutions from the factorization of an element of $G_D(\Q)$. Using these tools we can enumerate the number of normalized solution of the form $(a,b,c)$ for a given $c$.

\section{Cardinality of Rational Solutions}

With the two theorems from the previous section, we can count the the number of normalized solutions of the form $(a,b,c)$ with $c>0$. 

\begin{theorem}\label{thm:numsols}
Suppose that $-D\equiv 2,3$ (mod 4), $D > 1$ and $c = p_1^{n_1} \cdots p_k^{n_k}$. Then there is a normalized solution to \eqref{eq:pelleq} of the form $(a,b,c)$ if and only if 
all the $p_i\in S$. Further, if all the $p_i\in S$, we have exactly $2^{k-1}$ normalized solutions to \eqref{eq:pelleq} of the form $(a,b,c)$.
\end{theorem}

\begin{proof}
The first statement is proven by Lemma $\ref{lem:solexistence}$. Let us fix $c = p_1^{n_1} \cdots p_n^{n_k}$ such that $\leg{-D}{p_i} = 1$ for $1 \le i \le k$ and represent the solutions to $a^2 + Db^2 = c^2$ as $\frac{a + b\sqrt{-D}}{c}$. Define the sets $$T_1\ :=\ \left\{\frac{a + b\sqrt{-D}}{c} : a^2 + Db^2 = c^2, (a,b) = 1, c > 1\right\}$$ and $$T_2\ :=\ \left\{\pm \zeta_{p_1}^{\epsilon_1 n_1} \cdots \zeta_{p_k}^{\epsilon_k n_k} : \epsilon_i \in \{\pm1\}\right\}.$$ By Theorem \ref{thm:factorgdq}, we have $T_1 \subset T_2$ and, by Theorem \ref{thm:retrievesols}, we get $T_2 \subset T_1$. Now for every solution $(a,b)$ to $x^2 + Dy^2 = c^2$, we can find other solutions by multiplying $-1$ to $z = \frac{a + b\sqrt{-D}}{c}$ or by taking the complex conjugate of $c$. Therefore, for every $(a,b)$ there are four distinct solutions corresponding to the integers $a$ and $b$. They are $\frac{a + b\sqrt{-D}}{c}, \frac{-a + b\sqrt{-D}}{c}, \frac{a - b\sqrt{-D}}{c},$ and $\frac{-a - b\sqrt{-D}}{c}$. If $\Gamma$ is the abelian group of order $4$ generated by multiplication by $-1$ and complex conjugation that acts on $G_D(\Q)$, then the four solutions corresponding to the integers $a,b$ is the orbit of $\frac{a + b\sqrt{-D}}{c}$ under the action of $\Gamma$. Therefore, the normalized solutions to $x^2 + Dy^2 = c^2$ are given by $T_1/\Gamma$. 
    
As complex conjugation on an element $z\in T_2$ corresponds to the map $\epsilon_i \to -\epsilon_i$, so we get \be T_2/\Gamma \ = \  \left\{\zeta_{p_1}^{n_1} \zeta_{p_2}^{\epsilon_2 n_2} \cdots \zeta_{p_k}^{\epsilon_k n_k} : \epsilon_i \in \{\pm1\}\right\}.\ee  Since $T_2/\Gamma = T_1/\Gamma$, we have that the set of normalized solutions is given by $T_2/\Gamma$. Therefore, since there are $k-1$ numbers $\epsilon_i$ with two choices for each, we obtain $|T_2/\Gamma| = 2^{k-1}$.
\end{proof}

\begin{rek}\label{rek: evenc} While the theorems above are stated and proved for $-D\equiv 2,3$ (mod 4), we can generalize them to  $-D\equiv 1$ (mod 4) as well.  In that case all of these theorems remain valid \textit{only} when $c$ is odd, because when $-D\equiv 1$ (mod 4), there could exist normalized solutions of the form $(a,b,c)$ where $c$ is even, which cannot happen when $-D\equiv 2,3$ (mod 4) (see Lemma \ref{lem:oddsols}). \end{rek}

\begin{rek}\label{rek: D=1}
    For the case when $D = 1$, the argument must be modified, because in this case the group of units of $\Z[i]$ is $\{ \pm 1, \pm i\}$. Also for a solution $(a,b,c)$ to $x^2 + Dy^2 = z^2$ to be normalized when $D = 1$, there is the added condition that $a < b$. This is necessary, because if $(a,b,c)$ is a solution then so is $(b,a,c)$. The following proofs must be changed to account for these differences.
    
    Lemma $\ref{lem:definezeta}$ does not hold in this case, because it relies on the fact that the group of units of $\Z[\sqrt{-D}]$ is $\{\pm 1\}$. We use this lemma to prove that our definition of $\zeta_p$ is well defined. We can rectify this problem by defining $\zeta_p$ in the same way as \cite{Ye}. That is, for a prime $p$ such that $\leg{-1}{p} = 1$, $p = m^2 + n^2$ for $m, n \in \Z$. Because $p$ is odd, we have $|m|\ \ne |n|$. Therefore, we can assume $0 < m < n$. We can then define $q = m + ni$ and $\zeta_p = \frac{q}{\overline{q}}$ where $\overline{q}$ is the complex conjugate of $q$. Theorem $\ref{thm:factorgdq}$ and Theorem $\ref{thm:retrievesols}$ also rely on the fact that the group of units is $\{ \pm 1\}$. If we change the group of units to be $\{\pm 1, \pm i\}$, then the arguments hold if we change $z = \pm\zeta_{p_1}^{\pm \alpha_1} \cdots \zeta_{p_k}^{\pm \alpha_k}$ to $z = \pm i^r \zeta_{p_1}^{\pm \alpha_1} \cdots \zeta_{p_k}^{\pm \alpha_k}$ where $r \in \{0,1\}$. Therefore, we  still have the factorization $G_1(\Q) = U \times F$ where $U$ is the group of units of $\Z[i]$ and $F$ is the free abelian group with generators $\{\zeta_p\}$. 
    Finally, Theorem $\ref{thm:numsols}$ changes as follows. Given a solution $(a,b,c)$ we have 8 distinct solutions corresponding to the integers $a,b$. That is if we can multiply $\frac{a + bi}{c}$ by $\pm 1$, $\pm i$ or take complex conjugation to get another distinct solution. Therefore, we define $\Gamma$ to be the group generated by multiplication of $-1$, $i$ and complex conjugation. We also define $$T_2\ :=\ \left\{\pm i^r \zeta_{p_1}^{\epsilon_1 n_1} \cdots \zeta_{p_k}^{\epsilon_k n_k} : \epsilon_i \in \{\pm1\}, r \in \{0, 1\} \right\}.$$
    With these changes, the same argument gives us the result of Theorem $\ref{thm:numsols}$. 
\end{rek}

\section{Examples and Future Work}

\subsection{Examples}

We give a few examples of our results.\\ \

\begin{itemize}
    \item Theorem $\ref{thm:numsols}$ holds when the class group $C(-4D)$ has order $\le 2$ and $-D \equiv 2,3$ mod 4. This only occurs when $D=1,\ 2,\ 5,\ 6,\ 10,\ 13,\ 22,\ 37,\ 58$. \\ \
    
    \item The class group can also be a free $\Z_2$-module when $|C(-4D)| > 2$. For example, in case of $D=210$, the class number is $8$ and each reduced form has order at most $2$.\\ \

    \item The class group is not always a free $\Z_2$-module. For example, in the case of $D=26$, the class number is 6. So the class group cannot be a free $\Z_2$-module. In fact, even if the class number is $2^n$, the class group might not be a free $\Z_2$-module. For example, when $D=34$, the class number is $4$, and it contains a reduced form $[5,2,7]$, which does not have order at most $2$. The class containing this reduced form has order $4$, which means that the class group is a cyclic group of order 4. Note that in all of the above examples, $-D\equiv 2,3\ (\text{mod }4)$.\\ \
    
    \item For our results to be true, we need the class group to be a free $\Z_2$-module. Otherwise Theorem \ref{thm:numsols} might not be true. For example take $D=26$ as above and consider $c=5$. Note that \be \left(\frac{-26}{5}\right) \ = \ 1,\ee yet $x^2+26y^2=5^2$ has no normalized solution.\\ \
\end{itemize}

\subsection{Future Work}

Here are some possible avenues for extending the work of this paper.

\begin{itemize}
    \item This paper does not have full results for the case $-D \equiv 1$ (mod 4). As stated in Remark $\ref{rek: evenc}$, the problem with this case is that for a normalized solution $(a,b,c)$, $c$ can be even. Many of our results rely on the fact that certain primes are odd. What can be said about the solutions to $x^2 + Dy^2 = c^2$ when $c$ is even?\\ \
    
    \item Another restriction in this paper for the results to hold is that $C(-4D) \cong (Z_2)^n$. This is necessary in our method, because we use that every element of $C(-4D)$ has order at most $2$ when determining if there exists a normalized solution $(a,b,c)$ for a given $c$. With more advanced tools is there a way to extend these results for a broader range of integers $D$? \\ \
    
    \item What can be said about other Diophantine equations? Can we count the number of normalized solutions to $x^2 + y^2 + z^2 = w^2$, or even more generally $x^2 + D_1 y^2 + D_2 z^2 = w^2$ for example? For $x_1^2 + x_2^2 + x_3^2 + x_4^2 = z^2$ what are the consequences of the multiplicativity of the quaternion norm?
\end{itemize}



\ \\

\end{document}